\documentclass{amsart}
\usepackage{amsmath,amsthm}
\usepackage{amsfonts,amssymb}

\usepackage{enumerate}

\usepackage{hyperref}  
\newtheorem{theorem}{Theorem}
\newtheorem{Prop}[theorem]{Proposition}
\newtheorem{lemma}[theorem]{Lemma}

\numberwithin{theorem}{section}

\newtheorem{Cor}[theorem]{Corollary}

\theoremstyle{definition}
\newtheorem{Def}[theorem]{Definition}

\theoremstyle{remark}

\usepackage{ dsfont }
\usepackage{mathtools}

\def\Z{\mathds{Z} }
\def\N{\mathds{N} }
\def\C{\mathds{C} }

\DeclareMathOperator{\Ad}{Ad}
\DeclareMathOperator{\End}{End}
\DeclareMathOperator{\bessel}{\mathcal B_\lambda}
\DeclareMathOperator{\U}{U}
\DeclareMathOperator{\sgn}{sgn}
\DeclareMathOperator{\Fock}{\mc F}
\DeclareMathOperator{\pFock}{F}
\DeclareMathOperator{\Lie}{Lie}
\newcommand{\bfip}[1]{\left<{#1}\right>_\mathcal B}
\newcommand{\ip}[1]{\left<{#1}\right>}

\newcommand{\ipJ}[1]{\left({#1}\right)_J}

\newcommand{\ipS}[1]{\left({#1}\right)_S}

\newcommand{\pt}[1]{\partial_{#1}}
\newcommand{\mf}[1]{\mathfrak{#1}}
\newcommand{\ds}[1]{\mathds{#1}}
\newcommand{\mc}[1]{\mathcal{#1}}
\newcommand{\ol}[1]{\overline{#1}}

\newcommand{\g}{{\mathfrak{g}}}
\newcommand{\oa}{\bar{0}}
\newcommand{\ob}{\bar{1}}
\newcommand{\minus}{\scalebox{0.9}{{\rm -}}}
\newcommand{\plus}{\scalebox{0.6}{{\rm+}}}

\DeclareMathOperator{\rol}{d\rho}

\DeclarePairedDelimiter\abs{\lvert}{\rvert}%
\DeclarePairedDelimiter\norm{\lVert}{\rVert}%
\makeatletter
\let\oldabs\abs
\def\abs{\@ifstar{\oldabs}{\oldabs*}}
\let\oldnorm\norm
\def\norm{\@ifstar{\oldnorm}{\oldnorm*}}
\makeatother

\begin{document}
\title[A superunitary Fock model of $\ds D(2,1;\alpha)$]{A superunitary Fock model of the exceptional Lie supergroup $\ds D(2,1;\alpha)$}

\author{Sigiswald Barbier}
\address{Department of Electronics and Information Systems \\Faculty of Engineering and Architecture\\Ghent University\\Krijgslaan 281, 9000 Gent\\ Belgium.}
\email{Sigiswald.Barbier@UGent.be}

\author{Sam Claerebout}
\address{Department of Electronics and Information Systems \\Faculty of Engineering and Architecture\\Ghent University\\Krijgslaan 281, 9000 Gent\\ Belgium.}
\email{Sam.Claerebout@UGent.be}

\date{\today}

\begin{abstract}
We construct a Fock model of the minimal representation of the exceptional Lie supergroup $\ds D(2,1, \alpha)$. Explicit expressions for the action are given  by integrating to group level a Fock model of the Lie superalgebra    $D(2,1, \alpha)$ constructed earlier by the authors.  It is also shown that the representation is superunitary in the sense of de Goursac--Michel. 
\end{abstract}

\maketitle

% \tableofcontents

\section{Introduction}

The main result of this paper is the construction of a Fock model of the minimal representation of the Lie supergroup $\ds  D(2,1,\alpha)$. We do this by integrating to group level the representation of the Lie superalgebra $D(2,1,\alpha)$ considered in \cite{BC3}. In that sense, this paper can be seen as a sequel to \cite{BC3}. We also show that this Fock model is superunitary in the sense of \cite{dGM}.  

Minimal representations of Lie groups have a long tradition and can be constructed in many settings \cite{BrylinskiKonstant, DvorskySahi, GanSavin, HKM, KM, Torasso, VergneRossi}. In the philosophy of the orbit method, minimal representations correspond to the minimal nilpotent orbit of the coadjoint action of the Lie group on the dual Lie algebra. They are `small' infinite-dimensional representations, or more technically, they attain the smallest Gelfand--Kirillov dimension of all possible infinite-dimensional representations \cite{GanSavin}. This implies that there are a lot of symmetries in their realisations which leads to a rich representation theory. 

Recently, there has been an effort to generalize the framework of minimal representations  to the setting of Lie supergroups and Lie superalgebras \cite{BC1, BF, BCD, BC3}. 
Although this is a logical next step, there are lot of technical and conceptual hurdles. For instance, a lot of tools used for Lie groups are not yet developed or become much more complex in the super setting.
Another obstacle is the fact that  in \cite{NeebSalmasian} it is shown that there are no superunitary representations for a large class of Lie supergroups in the standard definition \cite{CCTV}  of a superunitary representation. This has lead to the development of alternative definitions of what should be a superunitary representation \cite{dGM, Tuynman, Tuynman2}. However, at the moment no satisfactory definition has been found. Therefore it is important to construct concrete models of representations that `ought' to be superunitary, as we will do in this paper. 

For the orthosymplectic Lie supergroup $OSp(p,q|2n)$ a Schr\"odinger model of the minimal representation was constructed in \cite{BF} using the framework of Jordan (super)algebras developed in \cite{HKM}. This generalizes the minimal representation of $O(p,q)$ considered in \cite{KM, KO1,KO2,KO3}.
Later, also a Fock model and intertwining Segal-Bargmann transform for  $OSp(p,q|2n)$ were obtained in \cite{BCD}.

Recently, a Schr\"odinger model, Fock model and Segal-Bargmann transform of the Lie superalgebra $D(2,1,\alpha)$ were constructed \cite{BC3}. The paper \cite{BC3} works entirely on algebra level. In particular it does not say anything about unitarity. It does, however, show that there exists a superhermitian product for which the Fock model is invariant. 
The goal of this paper is to integrate the Fock model considered in \cite{BC3} to group level. We will show that the superhermitian product can be extended to a Hilbert space and that our representation extend to a superunitary representation in the sense of \cite{dGM}. 

\subsection{Contents}
Let us now describe the contents of this paper. 
We start in Section \ref{D(2,1,alpha)} by recalling the definition of the Lie superalgebra $D(2,1,\alpha)$ and giving an explicit expression of the Fock model considered in \cite{BC3}. 
In Section \ref{Section group}, we introduce the Lie supergroup $\ds D(2,1,\alpha)$ and deduce some properties we need to integrate the representation, while in Section \ref{Section Fock space}, we recall the necessary properties of the polynomial Fock space considered in \cite{BC3} and complete it to a Hilbert superspace. 

Section \ref{Section superunitarity} contains the main content of this paper. 
We start by giving an explicit form of the representation of the Lie supergroup $\ds  D(2,1,\alpha)$ in Theorem \ref{ExplicitK}. We also give two alternatives way to present this representation (Corollaries \ref{ExplicitKActive} and \ref{ExplicitKl}). 
Note, however, that for one generating element of $\ds D(2,1,\alpha)$ we were only able to give an explicit form if $\alpha > 0$.

We recall the definition of a superunitary representation (SUR) as introduced in  \cite{dGM} in  Subsection \ref{subsecton superunitary} and show that the Fock model is such a SUR if $\alpha <0$ (Theorem \ref{ThSuperUnitary}).
In \cite{dGM}, also the concept of a strong SUR is defined. However, we show that the Fock model is never a strong SUR (Theorem \ref{Theorem not strong sur}).

\subsection{Notations}
The field $\ds K$ will always mean the real numbers $\ds R$ or the complex numbers $\C$. Function spaces will always be defined over $\C$. We use the convention $\ds N = \{0,1,2,\ldots\}$ and denote the complex unit by $\imath$.

A supervector space is a $\Z/2\Z$-graded vector space $V=V_{\ol 0}\oplus V_{\ol 1}$. An element $v\in V$ is called homogeneous if $v\in V_i$, $i\in \Z/2\Z$. We call $i$ its parity and denote it by $|v|$. When we use $|v|$ in a formula, we are considering homogeneous elements, with the implicit convention that the formula has to be extended linearly for arbitrary elements. If $\dim(V_i) = d_i$, then we write $\dim(V) = (d_{\ol 0}|d_{\ol 1})$. We denote the super-vector space $V$ with $V_{\ol 0} = \ds K^m$ and $V_{\ol 1} = \ds K^n$ as $\ds K^{m|n}$. A superalgebra is a supervector space $A=A_{\ol 0}\oplus A_{\ol 1}$ for which $A$ is an algebra and $A_iA_j\subseteq A_{i+j}$.

\section{The Lie superalgebra $D(2,1;\alpha)$} \label{D(2,1,alpha)}
\subsection{The construction of $D(2,1;\alpha)$}\label{ssConstruct}
We can deform $D(2,1)=\mathfrak{osp}(4|2)$ to obtain a one-parameter family of $(9|8)$-dimensional Lie superalgebras of rank $3$. We will define these Lie superalgebras in the same way as we did in \cite{BC3} using a construction of Scheunert. We will use the same notations as in \cite{Mu}.

Consider a two-dimensional vector space $V$ with basis $u_{+}$ and $u_{\minus}$. Introduce a non-degenerate skew-symmetric bilinear form $\psi$ by $\psi(u_{\plus},u_{\minus})=1$. 
We will need three copies $(V_i, \psi_i)$, $i=1,2,3$ of $(V,\psi)$ and the corresponding Lie algebra $\mathfrak{sl}(V_i)= \mathfrak{sp}(\psi_i)$ of linear transformations preserving $\psi_i$. 

We use the following data to define a Lie superalgebra:
\begin{itemize}
\item  a Lie algebra $\g_{\oa}$,
\item  a $\g_{\oa}$-module $\g_{\ob}$,
\item a $\g_{\oa}$-morphism $p:S^2(\g_{\ob})\rightarrow \g_{\oa}$, with $S^2(\g_{\ob})$ the symmetric tensor power,
\item for all $a,b,c \in \g_{\ob}$ the morphism $p$ satisfies \begin{eqnarray}
 [p(a,b),c]+[p(b,c),a]+[p(c,a),b]=0, \label{Jacobi identity}
\end{eqnarray} where we denoted the $\g_{\oa}$-action on $\g_{\ob}$ by $[\cdot,\cdot]$. 
\end{itemize}
Then $\g= \g_{\oa} \oplus \g_{\ob}$ is a Lie superalgebra \cite[Remark 1.5]{CW}. 

We set
\begin{eqnarray*}
\g_{\oa} &=& \mathfrak{sp}(\psi_1) \oplus \mathfrak{sp}(\psi_2)\oplus \mathfrak{sp}(\psi_3) \\
\g_{\ob} & = & V_1 \otimes V_2 \otimes V_3
\end{eqnarray*}
and define the action of $\g_{\oa}$ on $\g_{\ob}$ by the outer tensor product
\[
(A,B,C)\cdot  x\otimes y \otimes z = A (x)\otimes y \otimes z+ x\otimes B(y) \otimes z+x\otimes y \otimes C(z).
\]
The $\g_{\oa}$-morphism $p$ is given by 
 \begin{align*}
  p(x_1\otimes x_2\otimes x_3,y_1\otimes y_2\otimes y_3)&= \sigma_1 \psi_2(x_2,y_2)\psi_3(x_3,y_3) p_1(x_1,y_1) \\
  &\quad +\sigma_2 \psi_3(x_3,y_3)\psi_1(x_1,y_1) p_2(x_2,y_2)\\
  &\quad +\sigma_3 \psi_1(x_1,y_1)\psi_2(x_2,y_2) p_3(x_3,y_3),
  \end{align*}
where $\sigma_i \in \ds K$  and $p_i\colon V_i \times V_i \to \mathfrak{sp}(\psi_i)$ is defined by
\[ p_i(x,y)z= \psi_i(y,z) x- \psi_i(z,x)y. \] 
Then $\g= \g_{\oa} \oplus \g_{\ob}$ is a Lie superalgebra if the morphism $p$ satisfies the Jacobi identity (\ref{Jacobi identity}). This is the case if and only if $\sigma_1+\sigma_2+\sigma_3=0$, see \cite[Lemma 4.2.1]{Mu}. If we denote  $\g$ by $\Gamma(\sigma_1,\sigma_2,\sigma_3)$ then we have \[ \Gamma(\sigma_1,\sigma_2,\sigma_3)\cong \Gamma(\sigma_1',\sigma_2',\sigma_3')
\]
if and only if there is a non-zero scalar $c$ and a permutation $\pi$ of $(1,2,3)$ such that $\sigma_i'= c \sigma_{\pi(i)}$ \cite[Lemma 5.5.16]{Mu}. 

We set \[D(2,1;\alpha) := \Gamma\left(\frac{1+\alpha}{2},\frac{-1}{2},\frac{-\alpha}{2}\right). \]
 The Lie superalgebra $D(2,1;\alpha)$ is simple unless $\alpha=0$ or $\alpha= -1$. Furthermore, we have the isomorphism
 \[ D(2,1;\alpha)\cong D(2,1;\beta)\]
  if and only if $\alpha$ and $\beta$ are in the same orbit under the transformations $\alpha \mapsto \alpha^{-1}$ and $\alpha \mapsto -1-\alpha$.

Consider the  matrices
\[
E_i= \left(\begin{array}{rr}
0 & 1 \\
0 & 0
\end{array}\right), \quad F_i =  \left(\begin{array}{rr}
0 & 0 \\
1 & 0
\end{array}\right),\quad  H_i=\left(\begin{array}{rr}
1 & 0 \\
0 & -1
\end{array}\right).
\] 
They give a realisation of $\mathfrak{sl}(V_i)$ where
 the vector space $V_i$ is given by
\[
u_{\plus}^i = (1,0)^t , \quad u_{\minus}^i = (0,1)^t.
\]
We also obtain $p$ from
\[
p_i(u_{\plus}^i,u_{\plus}^i)= 2 E_i, \quad p_i(u_{\plus}^i,u_{\minus}^i)=-H_i, \quad p_i(u_{\minus}^i,u_{\minus}^i)= -2 F_i.
\]
For the odd basis elements $u_{\pm}^1 \otimes u_{\pm}^2 \otimes u_{\pm}^3$ of $D(2,1;\alpha)$ we introduce a more compact notation
\[
u_{\pm\pm\pm} := u_{\pm}^1 \otimes u_{\pm}^2 \otimes u_{\pm}^3.
\]

We have the following realisation of  $\mathfrak{sl}(2)$ in $D(2,1,\alpha)$
\begin{align*}
\{ E_{2} + E_{3}, H_{2}+ H_{3}, F_{2}+F_{3} \}.
\end{align*}
The corresponding three-grading by the eigenspaces of ad$(H_{2}+H_{3})$ is given by

\begin{align} \label{Three grading}
\begin{aligned}
\g_{\plus} &= \{ E_{3} ,  E_{2}, u_{\minus\plus\plus} ,u_{\plus\plus\plus}\} \\
\g_{\minus} &=\{ F_3, F_2, u_{\plus\minus\minus}, u_{\minus\minus\minus} \}  \\
\g_0 &=  \{   H_1, H_2,H_3, E_1, F_1, u_{\minus\plus\minus},u_{\plus\plus\minus},, u_{\plus\minus\plus},u_{\minus\minus\plus} \}.
\end{aligned}
\end{align}

\subsection{The Fock representation}
In \cite{BC3} a Fock representation $\rol$ depending on a parameter $\lambda\in\{1,\alpha\}$ was constructed on the so called polynomial Fock space $F_\lambda$. We will briefly reconstruct it here and refer to \cite[Section 4.3]{BC3} for further details.
Note that from now on we will exclude  $\alpha=0$ and $\alpha=1$ since for that case the picture becomes quite different. Remark that $\alpha=1$ correspond to the non-deformed case $D(2,1,1)=\mathfrak{osp}(4|2),$ while for $\alpha=0$, the algebra $D(2,1,0)$ is not simple.  See \cite[Section 4.2]{BC3} for a more detailed explanation.

Let $z_1,z_2$ and $z_3,z_4$ be the even resp. odd representatives of the coordinates on $\mc{P} (\mathds{C}^{2 |2})$. We define
\begin{align*}
V_\alpha &:=\{ a(2z_1z_2+z_3z_4) + b{z_2^2} + c z_2z_3 +d z_2z_4 \mid a,b,c,d \in \mathds{K} \} \subset \mathcal{P}_2 & \text{if } \lambda&=\alpha \text{ and}\\
V_1 &:=\{ a (2\alpha z_1z_2+z_3z_4) + b z_1^2 + c z_1z_3 + d z_1z_4 \mid a,b,c,d \in \mathds{K} \} \subset \mathcal{P}_2 & \text{if } \lambda&=1.
\end{align*}

\begin{Def}\label{DefFock}
Suppose $\lambda\in\{1,\alpha\}$, then the \textbf{polynomial Fock space} is defined as the superspace
\begin{align*}
F_\lambda := \mathcal{P}(\C^{2|2})/ \mc I_\lambda,
\end{align*}
with $\mc I_\lambda :=  \mc{P}(\ds \C^{2|2}) V_\lambda$.
\end{Def}
As shown in \cite[Section 5.1]{BC3}, if $p\in F_\alpha$, then there exists $p_{i,k}\in \ds C$ such that
\begin{align*}
p = p_{1,0}+\sum_{k=1}^{\infty}z_1^{k-1}(p_{1,k}z_1+p_{2,k}z_2+p_{3,k}z_3+p_{4,k}z_4). 
\end{align*}

The explicit expression for the Fock representation $\rol$ on $F_\lambda$ is as follows.

For $\mf g_{\minus}$ we obtain
\begin{align*}
\rol(F_2) &= -\dfrac{\imath}{2}(z_1 + \bessel(z_1)) -\frac{\imath}{2}(-\lambda +2z_1\pt {z_1} + z_3 \pt {z_3} + z_4\pt {z_4}),\\
\rol(F_3) &= -\dfrac{\imath}{2}(z_2 + \bessel(z_2)) -\frac{\imath}{2}(-\dfrac{\lambda}{\alpha} +2z_2\pt {z_2} + z_3 \pt {z_3} + z_4\pt {z_4}),\\
\rol(u_{\minus\minus\minus}) &= \dfrac{\imath}{2}(z_3 + \bessel(z_3))+ \frac{\imath}{2}(z_3\pt {z_1} + 2\alpha z_2\pt {z_4} +z_3\pt {z_2} + 2 z_1\pt {z_4}),\\
\rol(u_{\plus\minus\minus}) &= \dfrac{\imath}{4}(z_4 + \bessel(z_4))+ \frac{\imath}{4}(z_4\pt {z_1} - 2\alpha z_2\pt {z_3} +z_4\pt {z_2} - 2 z_1\pt {z_3}).
\end{align*}
For $\mf g_{\plus}$ we have
\begin{align*}
\rol(E_2) &=  -\dfrac{\imath}{2}(z_1 + \bessel(z_1)) +\frac{\imath}{2}(-\lambda +2z_1\pt {z_1} + z_3 \pt {z_3} + z_4\pt {z_4}),\\
\rol(E_3) &=  -\dfrac{\imath}{2}(z_2 + \bessel(z_2)) +\frac{\imath}{2}(-\dfrac{\lambda}{\alpha} +2z_2\pt {z_2} + z_3 \pt {z_3} + z_4\pt {z_4}),\\
\rol(u_{\minus\plus\plus}) &= -\dfrac{\imath}{2}(z_3 + \bessel(z_3)) +\frac{\imath}{2}(z_3\pt {z_1} + 2\alpha z_2\pt {z_4} +z_3\pt {z_2} + 2 z_1\pt {z_4}),\\
\rol(u_{\plus\plus\plus}) &= -\dfrac{\imath}{4}(z_4 + \bessel(z_4)) +\frac{\imath}{4}(z_4\pt {z_1} - 2\alpha z_2\pt {z_3} + z_4\pt {z_2} - 2 z_1\pt {z_3}).
\end{align*}
For $\mf g_0$ we have
\begin{gather*}
\rol(F_1) = 2z_3\pt {z_4}, \quad \rol(E_1) = 2^{-1}z_4\pt {z_3},\quad \rol(H_1) = z_4\pt {z_4}-z_3\pt {z_3},\\
\rol(H_2) = z_1-\bessel(z_1), \quad \rol(H_3) = z_2-\bessel(z_2),\\
\rol(u_{\minus\minus\plus}+u_{\minus\plus\minus}) = -(z_3-\bessel(z_3)), \quad \rol(u_{\plus\minus\plus}+u_{\plus\plus\minus}) = -2^{-1} (z_4-\bessel(z_4)),\\
\rol(u_{\minus\minus\plus}-u_{\minus\plus\minus}) =  -z_3\pt {z_1} - 2\alpha z_2\pt {z_4} +z_3\pt {z_2} + 2 z_1\pt {z_4},\\
\rol(u_{\plus\minus\plus}-u_{\plus\plus\minus}) =  2^{-1}(-{z_4}\pt {z_1} + 2\alpha z_2\pt {z_3} + z_4\pt {z_2} - 2 z_1\pt {z_3}).
\end{gather*}

Here $\bessel(z_i)$ denotes the Bessel operator of $z_i$ and are expicitly given by
\begin{align*}
\bessel(z_1) &= (-\lambda+z_1\pt {z_1} + z_3\pt {z_3} + z_4\pt {z_4})\pt {z_1} -2\alpha {z_2} \pt {z_3} \pt {z_4}, \\
\bessel(z_2) &= (-\dfrac{\lambda}{\alpha}+{z_2}\pt {z_2} + z_3\pt {z_3} + z_4\pt {z_4})\pt {z_2} -2 z_1 \pt {z_3} \pt {z_4},\\
\bessel(z_3) &= (-2\lambda + 2z_1\pt {z_1} + 2\alpha {z_2}\pt {z_2} + 2(1+\alpha)z_3\pt {z_3})\pt {z_4} + {z_3} \pt {z_1} \pt {z_2},\\
\bessel(z_4) &= (2\lambda - 2z_1\pt {z_1} - 2\alpha {z_2}\pt {z_2} - 2(1+\alpha)z_4\pt {z_4})\pt {z_3} + z_4 \pt {z_1} \pt {z_2}.
\end{align*}

\subsection{Additional representations}\label{ssAddRepr}

As mentioned in Section \ref{ssConstruct}, we have the isomorphisms $D(2,1;\alpha)\cong D(2,1;\beta)$ if and only if $\beta$ is in the same orbit as $\alpha$ under the transformations $\alpha \mapsto \alpha^{-1}$ and $\alpha \mapsto -1-\alpha$, i.e.,
\begin{align*}
\beta\in \{\alpha, -1-\alpha, -1-\alpha^{-1}, \alpha^{-1}, (-1-\alpha)^{-1}, (-1-\alpha^{-1})^{-1}\}
\end{align*}
These isomorphisms give rise to additional representations of $D(2,1;\alpha)$. A straightforward verification shows the isomorphism between $D(2,1;\alpha)$ and $D(2,1;\alpha^{-1})$ respects the three grading introduced in Equation \ref{Three grading}, while the isomorphism between $D(2,1;\alpha)$ and $D(2,1;-1-\alpha)$, in general, does not.

Let $\rol^{\alpha}_\lambda$ denote the Fock representation corresponding to $D(2,1;\alpha)$ with parameter $\lambda\in\{1,\alpha\}$. The isomorphism between $D(2,1;\alpha)$ and $D(2,1;\alpha^{-1})$ induces an equivalence between $\rol^\alpha_\lambda$ with parameter $\lambda=\alpha$ and $\rol^{\alpha^{-1}}_\lambda$ with $\lambda = 1$. Therefore, without loss of generality, we may choose $\lambda = \alpha$ and set $\rol^\alpha :=\rol^\alpha_\alpha$ . 

For an arbitrary $\alpha$ we now find that precomposing $\rol^\beta$ with the isomorphism $D(2,1;\alpha) \to D(2,1;\beta)$ for all possible values of $\beta$ gives us
\begin{align*}
&&\rol^\alpha, &&\rol^{-1-\alpha}, &&\rol^{-1-\alpha^{-1}} &&\rol^{\alpha^{-1}}, &&\rol^{(-1-\alpha)^{-1}}, &&\rol^{(-1-\alpha^{-1})^{-1}},
\end{align*}
which are all possibly distinct representations of $D(2,1;\alpha)$.

\section{The Lie supergroup $\ds D(2,1;\alpha)$}
\label{Section group}
In this section we define the supergroup $\ds D(2,1;\alpha)$ which has $D(2,1;\alpha)$ as its Lie superalgebra. We also give some basic results of $SL(V)$, which we will need later on. Note that in this section we will work  over the field $\ds R$ of real numbers. 

\subsection{Definition of $\ds D(2,1;\alpha)$}

We will use the characterisation of Lie supergroups based on pairs, see for example \cite[Chapter 7]{CCF} for more details.
\begin{Def}
A Lie supergroup $G$ is a pair $(G_0, \mf g)$ together with a morphism $\sigma : G_0\rightarrow \End(\mf g)$ where $G_0$ is a Lie group and $\mf g$ is a Lie superalgebra for which
\begin{itemize}
\item $\Lie(G_0)\cong \mf g_{\ol 0}$.
\item For all $g\in G_0$ we have $\left.\sigma(g)\right|_{\mf g_{\ol 0}} = \Ad(g)$, where $\Ad$ is the adjoint representation of $G_0$ on $\mf g_{\ol 0}$.
\item For all $X\in \mf g_{\ol 0}$ and $Y\in \mf g$ we have
\begin{align*}
d\sigma(X)Y= \left.\dfrac{d}{dt}\sigma(\exp(tX))Y\right|_{t=0}=[X,Y].
\end{align*}
\end{itemize}
Since $\sigma$ extends the adjoint representation of $G_0$ on $\mf g_{\ol 0}$ we call it the \textbf{adjoint representation} of $G_0$ on $\mf g$ and denote it by $\Ad$. 
\end{Def}
Note that these Lie supergroups are called super Harish-Chandra pairs in \cite{dGM}. The term Lie supergroup is then used for a supermanifold endowed with a group structure for which the multiplication is a smooth map. However, as is mentioned in \cite{dGM} these two structures are categorically equivalent.

Recall $\mf g = D(2,1;\alpha)$ and define $G_0 := SL(V_1)\times SL(V_2)\times SL(V_3)$, where $V_i$ is a copy of the two dimensional vector space $V$ with basis $u_+$ and $u_-$. Then $\ds D(2,1;\alpha) := (G_0, \mf g)$ is a Lie supergroup if we extend the adjoint representation as follows. For $A_i\in SL(V_i)$ and $x\otimes y \otimes z\in \mf g_{\ol 1} = V_1\otimes V_2\otimes V_3$ we define
\[
\Ad(A_1,A_2,A_3) x\otimes y \otimes z := A_1(x)\otimes A_2(y) \otimes A_3(z)
\]
and for $X_i \in \{H_i,E_i,F_i\}$ we define
\[
\Ad(A_1,A_2,A_3) X_i := A_iX_iA_i^{-1}
\]
and extend it linearly.

\subsection{Properties of $SL(V)$}

Define the following one-dimensional subgroups of $SL(V_i)$ for $i\in\{1,2,3\}$
\begin{align*}
K_i &:= \left\lbrace K_i(k_i) := \left(\begin{matrix}
\cos(k_i) & -\sin(k_i)\\
\sin(k_i) & \cos(k_i)
\end{matrix}\right) \mid k_i \in \ds R \right\rbrace,\\
 A_i &:= \left\lbrace A_i(a_i) := \left(\begin{matrix}
\exp(a_i) & 0\\
0 & \exp(-a_i)
\end{matrix}\right) \mid a_i \in \ds R \right\rbrace.
\end{align*}
 On the one hand we have the Cartan decomposition of $SL(V_i)$.
\begin{theorem}[Cartan decomposition]
We have a decomposition $SL(V) = KAK$, i.e., every $g\in \mf{sl}(V)$ can be written as $g = kak'$ with $k, k'\in K$ and $a\in A$.
\end{theorem}

This decomposition implies that a representation of $SL(V_i)$ is fully determined by its restriction to $K_i$ and $A_i$. On the other hand we have an explicit integration of $\mf{sl}(V_i)$ to $SL(V_i)$.

\begin{lemma}\label{anticompow}
Suppose $A$, $B$ and $C$ are three anticommuting variables. Then
\begin{align*}
(A+B+C)^{2j} = \sum_{a+b+c=j}\binom{j}{a,b,c}A^{2a}B^{2b}C^{2c},
\end{align*}
for all $j\in \ds N$.
\end{lemma}

\begin{proof}
This follows immediately from the multinomial theorem and the fact that $(A+B+C)^2 = A^2+B^2+C^2$ is a sum of three commuting variables.
\end{proof}

\begin{theorem}
Suppose $g\in SL(V_i)$. There exists an $X\in\Lie(SL(V_i))$ such that $g = \exp(X)$ if and only if $g$ is the identity or
\begin{align*}
g = \left(\begin{matrix}
\cosh(\rho)+a\rho^{-1}\sinh(\rho) & (l-k)\rho^{-1}\sinh(\rho)\\
(l+k)\rho^{-1}\sinh(\rho) & \cosh(\rho)-a\rho^{-1}\sinh(\rho)
\end{matrix}\right),
\end{align*}
for some $a,k,l\in \ds R$ such that $\rho := \sqrt{a^2+l^2-k^2}\neq 0$. In this case we have $X = k(F_i-E_i) + a H_i + l(F_i+E_i)$.
\end{theorem}
\begin{proof}
Any element $X\in \mf{sl}(V_i)$ can be written as $X=k(F_i-E_i) + a H_i + l(F_i+E_i)$ for $a,k,l\in \ds R$. We will calculate $\exp(X)$ explicitly. Note that $F_i-E_i$, $H_i$ and $F_i+E_i$ anticommute with each other and $(\imath(F_i-E_i))^2=H_i^2=(F_i+E_i)^2=I$. Using Lemma \ref{anticompow} we find
\begin{align*}
\exp(X) &= \sum_{j=0}^{\infty}\dfrac{1}{j!}(k(F_i-E_i) + a H_i + l(F_i+E_i))^j\\
&= \sum_{j=0}^{\infty}\dfrac{1}{(2j)!}(k(F_i-E_i) + a H_i + l(F_i+E_i))^{2j}\\
&\quad+ (k(F_i-E_i) + a H_i + l(F_i+E_i))\\
&\quad \times\sum_{j=0}^{\infty}\dfrac{1}{(2j+1)!}(k(F_i-E_i) + a H_i + l(F_i+E_i))^{2j}\\
&= \sum_{j=0}^{\infty}\sum_{u+v+w=j}\dfrac{k^{2u}a^{2v}l^{2w}j!}{(2j)!u!v!w!}(F_i-E_i)^{2u}H_i^{2v}(F_i+E_i)^{2w}\\
&\quad+ (k(F_i-E_i) + a H_i + l(F_i+E_i))\\
&\quad\times \sum_{j=0}^{\infty}\sum_{u+v+w=j}\dfrac{k^{2u}a^{2v}l^{2w}j!}{(2j+1)!u!v!w!}(F_i-E_i)^{2u}H_i^{2v}(F_i+E_i)^{2w}\\
&= I\sum_{j=0}^{\infty}\sum_{u+v+w=j}\dfrac{(\imath k)^{2u}a^{2v}l^{2w}j!}{(2j)!u!v!w!}\\
&\quad+ (k(F_i-E_i) + a H_i + l(F_i+E_i))\sum_{j=0}^{\infty}\sum_{u+v+w=j}\dfrac{(\imath k)^{2u}a^{2v}l^{2w}j!}{(2j+1)!u!v!w!}\\
&= \cosh(\rho)I + (k(F_i-E_i) + a H_i + l(F_i+E_i))\rho^{-1}\sinh(\rho),
\end{align*}
for $\rho\neq 0$. For $\rho=0$ this calculation gives us $\exp(X) =I$.
\end{proof}

Note that in particular, we have
\begin{align*}
K_i &= \{\exp(k_i(F_i-E_i)) \mid k_i \in \ds R \}, & A_i &= \{\exp(a_i H_i)\mid a_i \in \ds R  \}.
\end{align*}
This implies that from an explicit representation of $\mf{sl}(V_i)$ we can obtain an explicit action of elements in $K_i$ and $A_i$ when integrated to the group level. Because of the Cartan decomposition this then defines an action of $SL(V_i)$.

Since we can write every element of $SL(V_i)$ as a finite product of exponentials of elements of $\mf {sl}(V_i)$, we obtain the following corollary for $D(2,1;\alpha)$.
\begin{Cor}\label{CorProd}
Every element of $G_0=SL(V_1)\times SL(V_2)\times SL(V_3)$ can be written as a finite product of exponentials of elements of $\mf g_{\ol 0}$, i.e., for all $g \in G_0$ we have
\begin{align*}
g = \prod_{i=1}^n \exp(X_i),
\end{align*}
for some $X_i\in \mf g_{\ol 0}$ and $n\in \ds N$.
\end{Cor}

\section{The Fock space $\Fock$}
\label{Section Fock space}

In this section we introduce the notion of a Hilbert superspace as defined in \cite{dGM}. We also extend the polynomial Fock space $F_\lambda$ to the Fock space $\Fock$ and show it is such a Hilbert superspace when combined with the Bessel-Fischer product.

From now on we will restrict ourselves to the case $\alpha\in \ds R\setminus \ds N$ since only then the Bessel-Fischer product will be non-degenerate. Furthermore, we also choose $\lambda=\alpha$ and denote the polynomial Fock space $F_\lambda$ by $\pFock$. Recall from Subsection \ref{ssAddRepr} that the case $\lambda=1$ is always equivalent to a representation with $\lambda=\alpha$.

\subsection{The Bessel-Fischer product}
In \cite[Section 5]{BC3}, a non-degenerate, sesquilinear, superhermitian form on $\pFock$ was introduced. This product is a generalization of the Bessel-Fischer inner product on the polynomial space $\mc P(\C^m)$ considered in \cite[Section 2.3]{HKMO}.
\begin{Def}
For $p, q\in \pFock$ we define the \textbf{Bessel-Fischer product} of $p$ and $q$ as
\begin{align*}
\bfip{p,q} := \left. p(\bessel)\bar q(z)\right|_{z=0},
\end{align*}
where $\bar q(z) = \overline{q(\bar z)}$ is obtained by conjugating the coefficients of the polynomial $q$ and $p(\bessel)$ is obtained by replacing $z_i$ by $\bessel (z_i)$.
\end{Def}

From \cite[Proposition 5.6.]{BC3} we obtain the following explicit form of the Bessel-Fischer product.

\begin{Prop}\label{PropBFPexplicit}
Suppose $p,q\in \{z_1^{k}, z_1^kz_2,z_1^kz_3, z_1^kz_4\}$, with $k\in\N$. Then the only non-zero evaluations of $\bfip{p,q}$ are
\begin{align*}
\bfip{z_1^{k}, z_1^{k}} &= -\bfip{z_1^{k}z_2, z_1^{k}z_2} = k!(-\alpha)_{k},\\
\bfip{z_1^{k}z_3, z_1^{k}z_4} &= -\bfip{z_1^{k}z_4, z_1^{k}z_3}  = 2k!(-\alpha)_{k+1},
\end{align*}
where we used the Pochhammer symbol $(a)_k = a(a+1)(a+2)\cdots (a+k-1)$.
\end{Prop}

From this explicit form we can easily see that the Bessel-Fischer product is degenerate if and only if $\alpha\in \ds N$, which is why we assume $\alpha\in \ds R\setminus \ds N$.

\subsection{Definitions}

\begin{Def}
A \textbf{Hermitian superspace} $(\mc H, \left<\cdot\, ,\cdot \right>)$ is a supervector space $\mc H = \mc H_{\ol 0}\oplus \mc H_{\ol 1}$ endowed with a non-degenerate, superhermitian, sesquilinear form $\left<\cdot\, ,\cdot \right>$. If the inner product is a homogeneous form of degree $\sigma(\mc H)\in \Z/2\Z$, then $\mc H$ is called a Hermitian superspace of \textbf{parity} $\sigma(\mc H)$.
\end{Def}

According to the propositions in \cite[Section 5]{BC3}, the polynomial Fock space $\pFock$ endowed with the Bessel-Fischer product $\bfip{\cdot\, ,\cdot}$ is such a Hermitian superspace.

\begin{Def}
A \textbf{fundamental symmetry} of a Hermitian superspace $(\mc H, \left<\cdot\, ,\cdot \right>)$ is an endomorphism $J$ of $\mc H$ such that $J^4=1$, $\left<J(x) ,J(y) \right> = \left<x ,y \right>$ and $(\cdot \, , \cdot)_J$ defined by
\begin{align*}
\ipJ{x,y} := \left<x,J(y)\right>,
\end{align*}
for all $x,y\in \mc H$ is an inner product on $\mc H$.
\end{Def}

For $\pFock$ we find the following condition on its fundamental symmetries with respect to the Bessel-Fischer product.

\begin{Prop}\label{PropFunSym}
For all fundamental symmetries of $\pFock$ we must have
\begin{align*}
J(z_1^k)_{z_1^k} &= \epsilon_{1,k}\sgn((-\alpha)_k), & J(z_1^kz_2)_{z_1^kz_2} &= -\epsilon_{2,k}\sgn((-\alpha)_k),\\
J(z_1^kz_3)_{z_1^kz_4} &= \epsilon_{3,k}\sgn((-\alpha)_{k+1}), & J(z_1^kz_4)_{z_1^kz_3} &= -\epsilon_{4,k}\sgn((-\alpha)_{k+1}),
\end{align*}
for all $k\in\ds N$. Here $J(a)_{b}$ denotes the coefficient of $b$ in $J(a)$ and $\epsilon_{i,k}>0$ for all $i\in \{1,2,3,4\}$.
\end{Prop}
\begin{proof}
Suppose $J$ is an arbitrary fundamental symmetry of $\pFock$, then we have
\begin{align*}
\ipJ{z_1^k, z_1^k} = \bfip{z_1^k, J(z_1^k)} = J(z_1^k)_{z_1^k}\bfip{z_1^k,z_1^k} = J(z_1^k)_{z_1^k}k!(-\alpha)_k >0,
\end{align*}
for all $k\in \ds N$. Therefore, $J(z_1^k)_{z_1^k} = \epsilon \sgn((-\alpha)_k)$ for an $\epsilon>0$. The other three cases are similar.
\end{proof}

Based on this condition, we define the endomorphism $S$ of $\pFock$ by the linear extension of
\begin{align*}
&S(z_1^k) := \sgn((-\alpha)_k)z_1^k, && S(z_1^kz_2) := -\sgn((-\alpha)_k)z_1^kz_2,\\
&S(z_1^kz_3) := \sgn((-\alpha)_{k+1})z_1^kz_4, && S(z_1^kz_4) := -\sgn((-\alpha)_{k+1})z_1^kz_3,
\end{align*}
for all $k\in\ds N$. Then, one can easily verify that $S$ is a fundamental symmetry of $\pFock$ with respect to the Bessel-Fischer product.

\begin{Prop}
Suppose $p,q\in \{z_1^{k}, z_1^kz_2,z_1^kz_3, z_1^kz_4\}$, with $k\in\N$. Then the only non-zero evaluations of $(p, q)_S$ are
\begin{align*}
(z_1^{k}, z_1^{k})_S &= (z_1^{k}z_2, z_1^{k}z_2)_S = k!|(-\alpha)_{k}|,\\
(z_1^{k}z_3, z_1^{k}z_3)_S &= (z_1^{k}z_4, z_1^{k}z_4)_S  = 2k!|(-\alpha)_{k+1}|,
\end{align*}
where we used the Pochhammer symbol $(a)_k = a(a+1)(a+2)\cdots (a+k-1)$.
\end{Prop}
\begin{proof}
This follows immediately from Proposition \ref{PropBFPexplicit}.
\end{proof}

\begin{Def}
A Hermitian superspace $(\mc H, \left<\cdot\, ,\cdot \right>)$ is a \textbf{Hilbert superspace} if there exists a fundamental symmetry $J$ such that $(\mc H, \ipJ{\cdot \, , \cdot})$ is a Hilbert space.
\end{Def}

Note that the choice of a fundamental symmetry does not matter for the topology, thanks to \cite[Theorem 3.4]{dGM}.

Denote by $\Fock$ the completion of $\pFock$ with respect to $\ipS{\cdot \, , \cdot}$, then $(\Fock, \bfip{\cdot\, ,\cdot})$ is a Hilbert superspace, which we call the Fock space. Define $\norm{f}_S := \sqrt{\ipS{f,f}}$, then we have
\begin{align*}
\Fock &= \left\lbrace f = f_{1,0} + \sum_{k=1}^\infty z_1^{k-1}(f_{1,k}z_1 + f_{2,k}z_2 + f_{3,k}z_3 + f_{4,k}z_4) : \norm{f}_S< \infty,\, f_{i,k} \in \ds C \right\rbrace.
\end{align*}
The condition $\norm{f}_S< \infty$ on $f$ is equivalent to the condition that the sums
\begin{align*}
\sum_{k=0}^\infty k!|(-\alpha)_k||f_{1,k}|^2, && \sum_{k=1}^\infty (k-1)!|(-\alpha)_{k-1}||f_{2,k}|^2,\\
\sum_{k=1}^\infty (k-1)!|(-\alpha)_{k}||f_{3,k}|^2, && \sum_{k=1}^\infty (k-1)!|(-\alpha)_{k}||f_{4,k}|^2
\end{align*}
converge.

\section{The superunitary representation $\rho_0$}
\label{Section superunitarity}

In this section we explicitly integrate the differential action $\rol$ of $D(2,1,\alpha)$ on $\pFock$ to an action $\rho_0$ of $\ds D(2,1,\alpha)$ on $\Fock$. We also introduce the notion of superunitary representations as defined in \cite{dGM}. Then, we prove that our action defines a superunitary representation on $\Fock$ for $\alpha<0$.

Recall from Section \ref{Section Fock space} that we assume $\alpha\in \ds R\setminus \ds N$. 

\subsection{Definition and explicit form}

We define $\rho_0(\exp(X)) := \exp(\rol(X))$ for all $X\in \mf g_{\ol 0}$. Because of Corollary \ref{CorProd} this defines a representation of all of $G_0$. We will now describe this representation more explicitly. Note that we omit the action of $A_2(a_2)$ from our explicit representation. This case will be discussed in Section \ref{SSActA2}.

\begin{theorem}\label{ExplicitK}
The representation $\rho_0$ acting on $f = f(z_1,z_2,z_3,z_4)\in \Fock$ is given by\begingroup\allowdisplaybreaks
\begin{align}\label{EqK1}
\rho_0(K_1(k_1))f &= f(z_1,z_2, \cos(k_1)z_3-2^{-1}\sin(k_1)z_4, 2\sin(k_1)z_3+ \cos(k_1)z_4),\\\label{EqK2}
\rho_0(K_2(k_2))f &= \exp(\imath \alpha k_2)f(\exp(-2\imath k_2)z_1,z_2, \exp(-\imath k_2)z_3, \exp(-\imath k_2)z_4),\\\label{EqK3}
\rho_0(K_3(k_3))f &= \exp(\imath k_3)f(z_1, \exp(-2\imath k_3)z_2, \exp(-\imath k_3)z_3, \exp(-\imath k_3)z_4),\\\label{EqA1}
\rho_0(A_1(a_1))f &= f(z_1, z_2, \exp(-a_1)z_3, \exp(a_1) z_4),\\\label{EqA3}
\rho_0(A_3(a_3))f &= (\cosh(a_3)+\sinh(a_3)z_2)\\\nonumber
&\times f(z_1,\tanh(a_3)+\cosh(a_3)^{-2}z_2,\cosh(a_3)^{-1}z_3,\cosh(a_3)^{-1}z_4),
\end{align}\endgroup
\end{theorem}

\begin{proof}
\item[(\ref{EqK1})] We have
\begin{align*}
\rho_0(K_1(k_1)) &= \exp(\rol(k_1(F_1-E_1))) = \exp(k_1(2z_3\pt{z_4}-\dfrac{1}{2}z_4\pt{z_3}))\\
&= \sum_{i=0}^\infty \dfrac{k_1^i}{i!} (2z_3\pt{z_4}-\dfrac{1}{2}z_4\pt{z_3})^i,
\end{align*}
with
\begin{align*}
(2z_3\pt{z_4}-\dfrac{1}{2}z_4\pt{z_3})^2 &= - (z_3\pt{z_3}+z_4\pt{z_4}),\\
(2z_3\pt{z_4}-\dfrac{1}{2}z_4\pt{z_3})^3 &= - (2z_3\pt{z_4}-\dfrac{1}{2}z_4\pt{z_3}),
\end{align*}
and therefore
\begin{align*}
\rho_0(K_1(k_1)) &= 1-(z_3\pt{z_3}+z_4\pt{z_4})+\sum_{i=0}^\infty (-1)^i\dfrac{k_1^{2i}}{(2i)!}(z_3\pt{z_3}+z_4\pt{z_4})\\
&\quad + \sum_{i=0}^\infty (-1)^i\dfrac{k_1^{2i+1}}{(2i+1)!}(2z_3\pt{z_4}-\dfrac{1}{2}z_4\pt{z_3})\\
&= 1-z_3\pt{z_3}-z_4\pt{z_4}+\cos(k_1)(z_3\pt{z_3}+z_4\pt{z_4})\\
&\quad + \sin(k_1)(2z_3\pt{z_4}-\dfrac{1}{2}z_4\pt{z_3}).
\end{align*}
This gives us
\begin{align*}
\rho_0(K_1(k_1))f(z_1,z_2,z_3,z_4) &= f(z_1,z_2, \left(\begin{matrix}
\cos(k_1) & -2^{-1}\sin(k_1)\\
2\sin(k_1) & \cos(k_1)
\end{matrix}\right)\binom{z_3}{z_4}).
\end{align*}
\item[(\ref{EqK2}) and (\ref{EqK3})]We have
\begin{align*}
\rho_0(K_2(k_2)) &=\exp(\rol(k_2(F_2-E_2)))\\
&= \exp(\alpha\imath  k_2 - 2\imath k_2 z_1\pt{z_1} - \imath k_2 z_3\pt{z_3} - \imath k_2 z_4\pt{z_4})\\
&= \exp(\alpha \imath k_2) \exp(-2\imath k_2 z_1\pt{z_1})\exp(-\imath k_2 z_3\pt{z_3})\exp(-\imath k_2 z_4\pt{z_4}),
\end{align*}
and
\begin{align*}
\rho_0(K_3(k_3)) &=\exp(\rol(k_3(F_3-E_3)))\\
&= \exp(\imath  k_3 - 2\imath k_3 z_2\pt{z_2} - \imath k_3 z_3\pt{z_3} - \imath k_3 z_4\pt{z_4})\\
&= \exp(\imath k_3) \exp(-2\imath k_3 z_2\pt{z_2})\exp(-\imath k_3 z_3\pt{z_3})\exp(-\imath k_3 z_4\pt{z_4}).
\end{align*}
Since $\exp(a z_i\pt{z_i})f(z_i) = f(\exp(a)z_i)$ for all $a\in \ds C$ we get
\begin{align*}
\rho_0(K_2(k_2))f(z_1,z_2,z_3,z_4) &= e^{\imath \alpha k_2}f(e^{-2\imath k_2}z_1, z_2, e^{-\imath k_2}z_3,e^{-\imath k_2}z_4),
\end{align*}
and
\begin{align*}
\rho_0(K_3(k_3))f(z_1,z_2,z_3,z_4) &= e^{\imath k_3}f(z_1, e^{-2\imath k_3}z_2, e^{-\imath k_3}z_3,e^{-\imath k_3}z_4),
\end{align*}
respectively.
\item[(\ref{EqA1})] We have
\begin{align*}
\rho_0(A_1(a_1))f &= \exp(a_1\rol(H_1))f = \exp(a_1(z_4\pt{z_4}-z_3\pt{z_3}))f\\
&= f(z_1, z_2, \exp(-a_1)z_3, \exp(a_1) z_4)
\end{align*}

\item[(\ref{EqA3})] We have
\begin{align*}
\rho_0(A_3(a_3))f &= \exp(a_3\rol(H_3))f = \exp(a_3(z_2+\pt{z_2}))f = \sum_{i=0}^\infty\dfrac{a_3^i}{i!}(z_2+\pt{z_2})^i f
\end{align*}
with
\begin{align*}
(z_2+\pt{z_2})^{2i}f &= \sum_{k=0}^{\infty} z_1^{k-1}(f_{1,k}z_1+f_{2,k}z_2),\\
(z_2+\pt{z_2})^{2i-1}f &= \sum_{k=0}^{\infty} z_1^{k}f_{2,k+1} + \sum_{k=0}^{\infty} z_1^{k-1}z_2 f_{1,k-1}\\
&= \sum_{k=0}^\infty z_1^{k-1}(f_{2,k+1}z_1+f_{1,k-1}z_2),
\end{align*}
for $i\geq 1$. Therefore
\begin{align*}
\rho_0(A_3(a_3))f &= f +(\cosh(a_3)-1) \sum_{k=0}^{\infty} z_1^{k-1}(f_{1,k}z_1+f_{2,k}z_2)\\
&\quad + \sinh(a_3)\sum_{k=0}^\infty z_1^{k-1}(f_{2,k+1}z_1+f_{1,k-1}z_2)\\
&= \sum_{k=0}^\infty z_1^{k-1} ((\cosh(a_3)f_{1,k}+\sinh(a_3)f_{2,k+1})z_1\\
&\quad + ( \cosh(a_3)f_{2,k} + \sinh(a_3)f_{1,k-1} )z_2 +f_{3,k}z_3+f_{4,k}z_4)\\
&= (\cosh(a_3)+\sinh(a_3)z_2)\\
&\quad \times f(z_1,\tanh(a_3)+\cosh(a_3)^{-2}z_2,\cosh(a_3)^{-1}z_3,\cosh(a_3)^{-1}z_4)
\end{align*}
\end{proof}

We have two alternative ways to present this representation. The first one is as follows. Suppose $f\in \Fock$ and define
\begin{align*}
f_1(z_1) := \sum_{k=0}^\infty z_1^{k}f_{i,k}\quad \text{ and }\quad f_i(z_1) := \sum_{k=0}^\infty z_1^{k}f_{i,k+1},
\end{align*}
for $i\in\{2,3,4\}$. Then we have $f = f_1(z_1)+f_2(z_1)z_2+f_3(z_1)z_3+f_4(z_1)z_4$ and we can view $f$ as the vector
\begin{align*}
f = \left(\begin{matrix}
f_1(z_1)\\
f_2(z_1)\\
f_3(z_1)\\
f_4(z_1)
\end{matrix} \right).
\end{align*}

The representation $\rho_0$ on $\Fock$ can now be given by matrices acting on $f\in\Fock$.

\begin{Cor}\label{ExplicitKActive}
The representation $\rho_0$ acting on $f\in\Fock$ given by\begingroup\allowdisplaybreaks
\begin{align*}
\rho_0(K_1(k_1))f &= \left(\begin{matrix}
1 & 0 & 0 & 0\\
0 & 1 & 0 & 0\\
0 & 0 & \cos(k_1) & 2\sin(k_1)\\
0 & 0 & -2^{-1}\sin(k_1) & \cos(k_1)
\end{matrix}\right)f,\\
\rho_0(K_2(k_2))f &= \left(\begin{matrix}
e^{\imath k_2(\alpha-2\ds E)} & 0 & 0 & 0\\
0 & e^{\imath k_2 (\alpha-2\ds E)} & 0 & 0\\
0 & 0 & e^{\imath k_2 (\alpha-1-2\ds E)} & 0\\
0 & 0 & 0 & e^{\imath k_2 (\alpha-1-2\ds E)}
\end{matrix}\right)f,\\
\rho_0(K_3(k_3))f &= \left(\begin{matrix}
e^{\imath k_3} & 0 & 0 & 0\\
0 & e^{-\imath k_3 } & 0 & 0\\
0 & 0 & 1 & 0\\
0 & 0 & 0 & 1
\end{matrix}\right)f,\\
\rho_0(A_1(a_1))f &= \left(\begin{matrix}
1 & 0 & 0 & 0\\
0 & 1 & 0 & 0\\
0 & 0 & e^{-a_1} & 0\\
0 & 0 & 0 & e^{a_1}
\end{matrix}\right)f,\\
\rho_0(A_3(a_3))f &= \left(\begin{matrix}
\cosh(a_3) & \sinh(a_3) & 0 & 0\\
\sinh(a_3) & \cosh(a_3)  & 0 & 0\\
0 & 0 & 1 & 0\\
0 & 0 & 0 & 1\\
\end{matrix}\right)f,
\end{align*}
where $\ds E := z_1\pt{z_1}$ denotes the Euler operator on $f_i(z_1)$, $i\in\{1,2,3,4\}$.\endgroup
\end{Cor}

The second method is as follows. Denote by $\mc P_k(\ds C^{m|n})$ the space of homogeneous superpolynomials of degree $k$ in $m$ even variables and $n$ odd variables. Then
\begin{align*}
\begin{matrix}
\phi : &F_\lambda &\rightarrow &\mc P_{\text{even}}(\ds C^{1|2}):= \bigoplus\limits_{k=0}^{\infty} \mc P_{2k}(\ds C^{1|2})\\
 &(z_1,z_2,z_3,z_4)&\mapsto &(2^{-1}\ell_1^2, \ell_2\ell_3,\ell_1\ell_3,\ell_1\ell_2),
\end{matrix}
\end{align*}
defines an isomorphism between $F_\lambda$ and the space of even degree superpolynomials in the even variable $\ell_1$ and the two odd variables $\ell_2, \ell_3$. Here the ``even'' in $\mc P_{\text{even}}(\ds C^{1|2})$ refers to the degree and not the parity of the superpolynomial terms.

The representation $\rho_0$ on $\Fock$ can now be given as an action on $f(\ell_1,\ell_2,\ell_3) \in \phi(\Fock)$.

\begin{Cor}\label{ExplicitKl}
The representation $\rho_0$ acting on $f = f(\ell_1,\ell_2,\ell_3) \in \phi(\Fock)$ is given by\begingroup\allowdisplaybreaks
\begin{align*}
\rho_0(K_1(k_1))f &= f(\ell_1, \cos(k_1)\ell_2+ 2\sin(k_1)\ell_3, -2^{-1}\sin(k_1)\ell_2+\cos(k_1)\ell_3),\\
\rho_0(K_2(k_2))f &= \exp(\imath \alpha k_2)f(\exp(-\imath k_2)\ell_1, \ell_2, \ell_3),\\
\rho_0(K_3(k_3))f &= \exp(\imath k_3)f(\ell_1, \exp(-\imath k_3)\ell_3, \exp(-\imath k_3)\ell_4),\\
\rho_0(A_1(a_1))f &= f(\ell_1, \exp(a_1)\ell_2, \exp(-a_1) \ell_3),\\
\rho_0(A_3(a_3))f &= (\cosh(a_3)+\sinh(a_3)\ell_2\ell_3)f(\ell_1, \cosh(a_3)^{-1}\ell_2, \cosh(a_3)^{-1}\ell_3)\\
&\quad+\sinh(a_3)(f(\ell_1,1,1)-f(\ell_1,1,0) - f(\ell_1, 0,1) + f(\ell_1,0,0)\\
&\quad+\tanh(a_3)(f(\ell_1,\ell_2,\ell_3)-f(\ell_1,\ell_2,0) - f(\ell_1, 0,\ell_3) + f(\ell_1,0,0))).
\end{align*}\endgroup
Note that the symbolic change of odd variables $\ell_2$ and $\ell_3$ to the constant $1$ is only well defined if we use the convention that every instance of $\ell_3\ell_2$ in $f$ is first rewritten as $-\ell_2\ell_3$.
\end{Cor}

\subsection{The action of $A_2(a_2)$}\label{SSActA2}

For the action $\rho_0(A_2(a_2))$ we were unable to find an explicit form if $\alpha<0$.

For $\alpha>0$ we can write $\Fock$ in terms of a Generalised Laguerre polynomial basis,
\begin{align*}
\Fock &= \left\lbrace g = \exp(-z_1)\left(\sum_{k=0}^\infty (g_{1,k}+g_{2,k}z_2)L_k^{(-1-\alpha)}(2z_1)\right.\right.\\
&\quad \left.\left. + \sum_{k=1}^\infty (g_{3,k}z_3 + g_{4,k}z_4)L_{k-1}^{(-\alpha)}(2z_1)\right) : \norm{g}_S< \infty,\, g_{i,k} \in \ds C \right\rbrace,
\end{align*}
Here
\begin{align*}
L_k^{(a)}(2x) = \dfrac{(-1)^k}{k!}\U(-k, a+1, 2x) = \dfrac{(-1)^k}{k!}\sum_{i=0}^k \dfrac{(-1)^i}{i!}(-a-k)_i (-k)_i (2x)^{k-i}
\end{align*}
are the generalised Laguerre polynomials and $\U(a,b,c)$ is the confluent hypergeometric function of the second kind. Note that this does not define a basis of $\Fock$ if $\alpha<0$, since then $\norm{\exp(-z_1)}_S \not< \infty$. We can now give the actions of $A_2(a_2)$ with respect to this basis.

\begin{Prop}\label{PropA2}
For $\alpha > 0$ we have
\begin{align*}
\rho_0(A_2(a_2))g(z) &= \exp(-z_1)\left(\sum_{k=0}^\infty \exp(a_2(2k-\alpha))(g_{1,k}+g_{2,k}z_2)L_k^{(-1-\alpha)}(2z_1))\right.\\\nonumber
&\quad \left. + \sum_{k=1}^\infty \exp(a_2(2k-\alpha-1))(g_{3,k}z_3 + g_{4,k}z_4)L_{k-1}^{(-\alpha)}(2z_1))\right).
\end{align*}
\end{Prop}

\begin{proof}
We have
\begin{align*}
\rho_0(A_2(a_2)) &= \exp(a_2\rol(H_2)) = \exp(a_2(z_1-\bessel(z_1)))\\
&= \exp(a_2(z_1+(\alpha-z_1\pt{z_1}-z_3\pt{z_3}-z_4\pt{z_4})\pt{z_1}))\\
&= \sum_{i=0}^\infty \dfrac{a_2^i}{i!}D^i,
\end{align*}
with
\begin{align*}
D = z_1+(\alpha-z_1\pt{z_1}-z_3\pt{z_3}-z_4\pt{z_4})\pt{z_1}.
\end{align*}
Since
\begin{align*}
D(\exp(-z_1)L_k^{(-1-\alpha)}(2z_1)) = (2k-\alpha)\exp(-z_1)L_k^{(-1-\alpha)}(2z_1)
\end{align*}
and
\begin{align*}
D(\exp(-z_1)z_jL_{k-1}^{(-\alpha)}(2z_1)) = (2k-\alpha-1)\exp(-z_1)z_jL_{k-1}^{(-\alpha)}(2z_1),
\end{align*}
for $j\in\{3,4\}$, we obtain the desired result.
\end{proof}

Despite not having an explicit form $\alpha <0$, we can show that this action is unitary if and only if $\alpha <0$.

\begin{Prop}\label{Prop_A2_Unitary}
The action $\rho_0(A_2(a_2))$ is a unitary operator on $(\Fock, \bfip{\cdot\, ,\cdot})$ for all $a_2\in \ds R$ if and only if $\alpha <0$.
\end{Prop}
\begin{proof}
First assume $\alpha>0$. From Proposition \ref{PropA2} we see that the eigenvalues of $\rho_0(A_2(a_2))$ are of the form $\exp(a)$, with $a\in \ds R$. Since these eigenvalues are not roots of unity, $\rho_0(A_2(a_2))$ can not be a unitary operator on $(\Fock, \bfip{\cdot\, ,\cdot})$.

Now assume $\alpha < 0$. In this case, we can easily see that the Fundamental symmetry $S$ commutes with $\rho_0(A_2(a_2))$. Because of \cite[Proposition 6.3]{BC3} we have
\begin{align*}
\bfip{\rol(H_2)p,q} = - \bfip{p, \rol(H_2)q},
\end{align*}
for $p,q\in \pFock$. This implies
\begin{align*}
(\rho_0(A_2(a_2))p,q)_S &= \bfip{\rho_0(A_2(a_2))p,S(q)} = \bfip{\exp(a_2\rol(H_2))p,S(q)}\\
&= \bfip{p,\exp(-a_2\rol(H_2))S(q)} =\bfip{p,\rho_0(A_2(-a_2))S(q)}\\
&= \bfip{p,S(\rho_0(A_2(-a_2))q)} = (p,\rho_0(A_2(-a_2))q)_S,
\end{align*}
for $p,q\in \pFock$, i.e., $\rho_0(A_2(a_2))$ acts as a unitary operator when acting on $\pFock$. Since $\pFock$ is dense in $\Fock$, we are finished.
\end{proof}

\subsection{Superunitary representations}
\label{subsecton superunitary}

The following definitions can be found in \cite{dGM}.

\begin{Def}
Let $(\mc H_1, \ip{\cdot\, , \cdot}_1)$ and $(\mc H_2, \ip{\cdot\, , \cdot}_2)$ be Hilbert superspaces and suppose $T: \mc H_1 \rightarrow \mc H_2$ is a linear operator. We call $T$ a \textbf{bounded operator} between $\mc H_1$ and $\mc H_2$ if it is continuous with respect to their Hilbert topologies. The set of bounded operators is denoted by $\mc B(\mc H_1, \mc H_2)$ and $\mc B(\mc H_1) := \mc B(\mc H_1, \mc H_1)$.
\end{Def}

\begin{Def}
Let $(\mc H_1, \ip{\cdot\, , \cdot}_1)$ and $(\mc H_2, \ip{\cdot\, , \cdot}_2)$ be Hilbert superspaces and suppose $T\in \mc B(\mc H_1, \mc H_2)$. The \textbf{superadjoint} of $T$ is the operator $T^\dagger\in \mc B(\mc H_2, \mc H_1)$ such that
\begin{align*}
\ip{T^\dagger(x),y}_1 = (-1)^{|T||x|}\ip{x, T(y)}_2,
\end{align*}
for all $x\in \mc H_2$, $y\in \mc H_1$.
\end{Def}

\begin{Def}
Let $(\mc H_1, \ip{\cdot\, , \cdot}_1)$ and $(\mc H_2, \ip{\cdot\, , \cdot}_2)$. A \textbf{superunitary operator} between $\mc H_1$ and $\mc H_2$ is a homogeneous operator $\psi\in\mc B(\mc H_1, \mc H_2)$ of degree $0$ satisfying $\psi^\dagger\psi = \psi\psi^\dagger = \ds 1$. The set of superunitary operators is denoted by $\mc U(\mc H_1, \mc H_2)$ and $\mc U(\mc H_1) := \mc U(\mc H_1, \mc H_1)$.
\end{Def}

\begin{Def}\label{Defsuperunitary}
A \textbf{superunitary representation} of a Lie supergroup $G=(G_0,\mf g)$ is a triple $(\mc H, \pi_0, d\pi)$ such that
\begin{itemize}
\item $\mc H$ is a Hilbert superspace.
\item $\pi_0:G_0\rightarrow \mc U(\mc H)$ is a group morphism.
\item For all $v\in \mc H$, the maps $\pi_0^v : g \mapsto \pi_0(g)v$ are continuous on $G_0$.
\item $d\pi : \mf g \rightarrow \End(\mc H^{\infty})$ is a $\ds R$-Lie superalgebra morphism such that $d\pi = d\pi_0$ on $\mf g_{\ol 0}$, $d\pi$ is skew-supersymmetric with respect to $\left< \cdot \, ,\cdot\right>$ and
\begin{align*}
\pi_0(g)d\pi(X)\pi_0(g)^{-1} = d\pi(\Ad(g)(X)),\quad \text{ for all } g\in G_0 \text{ and } X\in \mf g_{\ol 1}.
\end{align*}
Here $\mc H^\infty$ is the space of smooth vectors of the representation $\pi_0$ and $\Ad$ is the adjoint representation of $G_0$ on $\mf g$.
\end{itemize}
\end{Def}

Using this definition of a superunitary representation we can now prove the following result.

\begin{theorem}\label{ThSuperUnitary}
Assume $\alpha <0$. The triple $((\Fock, \bfip{\cdot \,, \cdot}), \rho_0, \rol)$ is a superunitary representation of $\ds D(2,1;\alpha)$.
\end{theorem}

\begin{proof}
Thanks to Corollary \ref{CorProd}, we only need to consider the representation $\rho_0$ on elements of the form $g=\exp(X_1)\cdots\exp(X_n)\in G_0$, with $X_i\in \mf{g}_{\ol 0}$ and $n\in \ds N$. We now prove the different conditions of Definition \ref{Defsuperunitary}.
\begin{itemize}
\item $(\Fock, \bfip{\cdot \,, \cdot})$ is a Hilbert superspace:\\
This follows from the definitions.
\item $\rho_0:G_0\rightarrow \mc U(\Fock)$ is a group morphism:\\
We wish to prove that $\rho_0(\exp(X_1)\cdots\exp(X_n))$ is a superunitary operator of $\Fock$. Because of \cite[Proposition 6.3]{BC3} we have
\begin{align*}
\bfip{\rho_0(\exp(X))p, q} =(-1)^{|X||p|} \bfip{p, \rho_0(\exp(-X))q},
\end{align*} 
which implies that the superadjoint of $\rho_0(\exp(X_1)\cdots\exp(X_n))$ is given by $\rho_0(\exp(-X_n)\cdots\exp(-X_1))$ and therefore it is a superunitary operator of $\Fock$.
\item For all $f\in \Fock$, the maps $\rho_0^f : g \mapsto \rho_0(g)f$ are continuous on $G_0$:\\
We need to prove the following
\begin{align*}
(\forall g\in G_0)(\forall \epsilon>0)(\exists U \text{ neighborhood of } g)(h\in U \implies ||\rho_0^f(g)-\rho_0^f(h)||_S<\epsilon).
\end{align*}
Since \[U_r := \left\lbrace\prod_{i=1}^3 K_i(k_i)A_i(a_i)K_i(k_i')g: \sum_{i=1}^{3} |k_i| + |a_i| + |k_i'| < r\right\rbrace\] is a neighbourhood of $g$ for all $r>0$, it suffices to prove
\begin{align*}
&(\forall g\in SL(V_i))(\forall \epsilon>0)(\exists \delta>0)(||\rho_0(X_i(\delta))f-f||_S<\epsilon),
\end{align*}
for $i\in\{1,2,3\}$ and $X_i \in \{K_i, A_i\}$.

For $A_2$ we know from Proposition \ref{Prop_A2_Unitary} that the actions is unitary if $\alpha <0$. Since unitarity implies continuity, we are done.

For $K_3$ we have
\begin{align*}
\rho_0(K_3(\delta))f &= e^{\imath \delta}\sum_{k=0}^\infty z_1^{k-1}(f_{1,k}z_1 + e^{-2\imath \delta}f_{2,k}z_2 + e^{-\imath \delta}(f_{3,k}z_3 + f_{4,k}z_4))\\
&=\sum_{k=0}^\infty z_1^{k-1}(e^{\imath \delta}f_{1,k}z_1 + e^{-\imath \delta}f_{2,k}z_2 +   f_{3,k}z_3 +   f_{4,k}z_4)
\end{align*}
and therefore
\begin{align*}
\norm{\rho_0(K_3(\delta))  f-  f}_S^2 &= \norm{\sum_{k=0}^\infty z_1^{k-1}((e^{\imath \delta}-1) f_{1,k}z_1 + (e^{-\imath \delta}-1) f_{2,k}z_2)}_S^2\\
&= (2-e^{\imath\delta}-e^{-\imath\delta})\\
&\quad \times\sum_{k=0}^{\infty} k!|(-\alpha)_k|| f_{1,k}|^2+(k-1)!|(-\alpha)_{k-1}|| f_{2,k}|^2,
\end{align*}
which goes to 0 as $\delta$ goes to $0$.

For $A_3$ we have
\begin{align*}
\rho_0(A_3(\delta))f -f &= \sum_{k=0}^\infty z_1^{k-1} (((\cosh(\delta)-1)f_{1,k}+\sinh(\delta)f_{2,k+1})z_1\\
&\quad + ( (\cosh(\delta)-1)f_{2,k} + \sinh(\delta)f_{1,k-1} )z_2)
\end{align*}
and therefore
\begin{align*}
\norm{\rho_0(A_3(\delta))  f-  f}_S^2 &= \sum_{k=0}^\infty |(\cosh(\delta)-1)f_{1,k}+\sinh(\delta)f_{2,k+1}|^2|(-\alpha)_k|k!\\
&\quad + |(\cosh(\delta)-1)f_{2,k} + \sinh(\delta)f_{1,k-1}|^2|(-\alpha)_{k-1}|(k-1)!\\
&\leq 2(\cosh(\delta)-1)^2\sum_{k=0}^\infty |f_{1,k}|^2|(-\alpha)_k|k!\\
&\quad +2\sinh(\delta)^2\sum_{k=0}^\infty |f_{2,k+1}|^2|(-\alpha)_k|k!\\
&\quad +2(\cosh(\delta)-1)^2 \sum_{k=1}^\infty|f_{2,k}|^2|(-\alpha)_{k-1}|(k-1)!\\
&\quad +2\sinh(\delta)^2\sum_{k=1}^\infty |f_{1,k-1}|^2|(-\alpha)_{k-1}|(k-1)!
\end{align*}
which goes to 0 as $\delta$ goes to $0$.

For $K_2$ we have
\begin{align*}
\rho_0(K_2(\delta))f -f &= \sum_{k=0}^{\infty}(e^{\imath \delta(\alpha-2k)}-1)(z_1^{k}f_{1,k} + z_2z_1^kf_{2,k+1})\\
&\quad + \sum_{k=0}^{\infty}(e^{\imath \delta(\alpha-2k-1)}-1)(z_3z_1^{k}f_{3,k+1} + z_4z_1^k f_{4,k+1})
\end{align*}
and therefore
\begin{align*}
&\norm{\rho_0(K_2(\delta))  f-  f}_S^2\\
&\quad= \sum_{k=0}^{\infty}2(1-\cos(\delta(\alpha-2k)))(|f_{1,k}|^2 + |f_{2,k+1}|^2)k!|(-\alpha)_k|\\
&\quad + \sum_{k=0}^{\infty}4(1-\cos(\delta(\alpha-2k-1)))(|f_{3,k+1}|^2 + |f_{4,k+1}|^2)k!|(-\alpha)_{k+1}|\\
&\quad\leq 4\sum_{k=0}^{\infty}(|f_{1,k}|^2 + |f_{2,k+1}|^2)k!|(-\alpha)_k| +8 \sum_{k=0}^{\infty}(|f_{3,k+1}|^2 + |f_{4,k+1}|^2)k!|(-\alpha)_{k+1}|\\
&\quad = 4\norm{f}_{S}^2
\end{align*}
Using Lebesgue's dominated convergence theorem we now find
\begin{align*}
&\lim_{\delta\rightarrow 0}\norm{\rho_0(K_2(\delta))  f-  f}_S^2\\
&\quad= \sum_{k=0}^{\infty}2\lim_{\delta\rightarrow 0}(1-\cos(\delta(\alpha-2k)))(|f_{1,k}|^2 + |f_{2,k+1}|^2)k!|(-\alpha)_k|\\
&\quad + \sum_{k=0}^{\infty}4\lim_{\delta\rightarrow 0}(1-\cos(\delta(\alpha-2k-1)))(|f_{3,k+1}|^2 + |f_{4,k+1}|^2)k!|(-\alpha)_{k+1}|\\
&\quad = 0,
\end{align*}
as desired.

Lastly, the $K_1$ and $A_1$ cases are analogous to the $K_3$ and $A_3$ cases.
\item $\rol : \mf g \rightarrow \End(\pFock)$ is a $\ds R$-Lie superalgebra morphism such that
\begin{itemize}
\item[(i)] $\rol = \rol_0$ on $\mf g_{\ol 0}$,
\item[(ii)] $\rol$ is skew-supersymmetric with respect to $\left< \cdot \, ,\cdot\right>$ and
\item[(iii)] $\rho_0(g)\rol(X)\rho_0(g)^{-1} = \rol(\Ad(g)(X))$, for all $g\in G_0$ and $X\in \mf g_{\ol 1}$:
\end{itemize}
Item (i) follows from
\begin{align*}
\rol_0(X)p = \left.\dfrac{d}{dt}\rho_0(\exp(tX))p\right|_{t=0} = \left.\dfrac{d}{dt}\exp(t\rol(X))p\right|_{t=0} = \rol(X)p,
\end{align*}
for all $p\in \pFock$ and $X\in \mf g_{\ol 0}$. Item (ii) follows directly from \cite[Proposition 6.3]{BC3} and item (iii) follows from
\begin{align*}
\rho_0(\exp(Y))\rol(X)\rho_0(\exp(Y))^{-1} &= \rho_0(\exp(Y))\rol(X)\rho_0(\exp(-Y))\\
&= \exp(\rol(Y))\rol(X)\exp(\rol(-Y))\\
&= \rol(\exp(Y)X\exp(-Y))\\
&= \rol(\Ad(\exp(Y))(X)),
\end{align*}
for all $X\in \mf g_{\ol 1}$ and $Y\in \mf g_{\ol 0}$.\qedhere
\end{itemize}
\end{proof}

The assumption $\alpha <0$ is only used to prove the continuity of $\rho_0(A_2(\delta))$. Note that Proposition \ref{Prop_A2_Unitary} only implies that the actions are not unitary if $\alpha>0$. It tells us nothing about the continuity in this case. It is possible that Theorem \ref{ThSuperUnitary} holds even without the assumption $\alpha<0$.

From the discussion in Section \ref{ssAddRepr}, we can at least conclude that for every $\alpha$ there always exists a superunitary representation of $\ds D(2,1;\alpha)$. Indeed, if $\alpha>0$, we can look at the Fock representation of $D(2,1;-1-\alpha)$ instead of the Fock representation of $D(2,1;\alpha)$.

\subsection{Strong superunitary representation}

In \cite[Section 4.4]{dGM} the notion of a strong superunitary representation is also defined. However, it is easy to prove that our superunitary representation is not a strong superunitary representation.

\begin{Def}\label{Defstrongsuperunitary}
A \textbf{strong superunitary representation} of a Lie supergroup $G=(G_0,\mf g)$ is a superunitary representation $(\mc H, \pi_0, d\pi)$ such that
\begin{itemize}
\item $(\mc H, \pi_0)$ is unitarizable,
\item $(\mc H, \pi_0, d\pi)$ admits a restriction to $(D(G_0),D(\mf g_{\ds R}))$.
\end{itemize}
Here $D(G_0)$ is the connected Lie subgroup of $G_0$ with Lie algebra $[(\mf g_{\ds R})_{\ol 1}, (\mf g_{\ds R})_{\ol 1}]$ and $D(\mf g_{\ds R}) := [(\mf g_{\ds R})_{\ol 1}, (\mf g_{\ds R})_{\ol 1}]\oplus (\mf g_{\ds R})_{\ol 1}$.
\end{Def}

\begin{theorem}
\label{Theorem not strong sur}
There does not exist a fundamental symmetry on $\pFock$ such that $((\Fock, \bfip{\cdot \,, \cdot}), \rho_0)$ is unitarizable. As a consequence $((\Fock, \bfip{\cdot \,, \cdot}), \rho_0, \rol)$ is not a strong superunitary representation.
\end{theorem}

\begin{proof}
Let $J$ be an arbitrary fundamental symmetry on $\pFock$. Thanks to Proposition \ref{PropFunSym} we may assume
\begin{align*}
J(z_1^k)_{z_1^k} &= \epsilon_{1,k}\sgn((-\alpha)_k), & J(z_1^kz_2)_{z_1^kz_2} &= -\epsilon_{2,k}\sgn((-\alpha)_k),\\
J(z_1^kz_3)_{z_1^kz_4} &= \epsilon_{3,k}\sgn((-\alpha)_{k+1}), & J(z_1^kz_4)_{z_1^kz_3} &= -\epsilon_{4,k}\sgn((-\alpha)_{k+1}),
\end{align*}
with $\epsilon_{i,k}>0$ for all $k\in\ds N$ and $i\in \{1,2,3,4\}$. Suppose $((\Fock, \bfip{\cdot \,, \cdot}), \rho_0)$ is unitarizable, then the inner product on $\pFock$ should be invariant under the derived action of $\rho_0$, i.e.,
\begin{align}\label{EqInv}
\ipJ{d\rho(X)p, q} = - \ipJ{p, d\rho(X)q}
\end{align}
for all $X\in D(2,1;\alpha)_{\ol{0}}$, $p,q\in \pFock$. Set $X = E_3+F_3$, $p=z_2$ and $q=1$. Then $d\rho(X) = -\imath(z_2-\pt{z_2})$ and the left hand side of equation (\ref{EqInv}) becomes
\begin{align*}
\ipJ{d\rho(X)p,q} = \bfip{\imath, J(1)} = \imath \epsilon_{1,0}
\end{align*}
while the right hand side becomes
\begin{align*}
- \ipJ{p, d\rho(X)q} = - \bfip{z_2, -\imath J(z_2)} = -(\ol{-\imath})\epsilon_{2,0} = -\imath \epsilon_{2,0},
\end{align*}
which implies equation (\ref{EqInv}) holds only if
\begin{align*}
\epsilon_{1,0} + \epsilon_{2,0} = 0.
\end{align*}
Since both $\epsilon_{1,0}$ and $\epsilon_{2,0}$ are greater than zero, this gives us a contradiction.
\end{proof}

\subsection{Harish-Chandra supermodules}
We will end this paper by giving an alternative, non-explicit, way to integrate the algebra representation of $D(2,1,\alpha)$ to group level. 
We do this by using the framework of Harish-Chandra supermodules developed in \cite{Alldridge}. It would be interesting to know if this abstract integration gives the same representation as the explicit integration of Theorem \ref{ExplicitK}, but we were unable to verify this. 

\begin{Def}\cite[Definition 4.1]{Alldridge} Let $V$ be a complex super-vector space, $G = (G_0, \mf g)$ a Lie supergroup and $K$ a maximal compact subgroup of $G_0$. Then $V$ is a \textbf{$(\mf g,K)$-module} if it is a locally finite $K$-representation that has also a compatible $\mf g$-module structure, that is, the derived action of $K$ agrees with the $\Lie(K)$-module structure:
\begin{align*}
d\pi_0(X)(v) = \left.\dfrac{d}{dt}\pi_0(\exp(tX))(v)\right|_{t=0} = d\pi(X)(v)\text{ for all } X\in \Lie(K),\, v\in V
\end{align*}
and
\begin{align*}
\pi_0(k)(d\pi(X)(v)) = d\pi(\Ad(k)(X))(\pi_0(k)(v)), \text{ for all } k\in K, X\in \mf g, v \in V,
\end{align*}
where $\pi_0$ is the $K$-representation and $d\pi$ the $\mf g$-representation. A $(\mf g,K)$-module is a \textbf{Harish–Chandra supermodule} if it is finitely generated over $U(\mf g)$ and is $K$-multiplicity finite.
\end{Def}

The maximal compact subgroup of $SL(V_i)$ is $K_i$. The maximal compact subgroup of $G_0$ is therefore the $3$-Torus $K :=  K_1\times  K_2\times K_3$.

\begin{Prop}\label{PropHC}
The module $\pFock$ is a Harish–Chandra supermodule.
\end{Prop}

\begin{proof}
That $\pFock$ is a $(\mf g, K)$-module follows from
\begin{align*}
\rol_0(X)(p) = \left.\dfrac{d}{dt}\rho_0(\exp(tX))(p)\right|_{t=0} = \left.\dfrac{d}{dt}\exp(t\rol(X))(p)\right|_{t=0} = \rol(X)p,
\end{align*}
and
\begin{align*}
\rho_0(\exp(Y))(\rol(X)p) &= \rho_0(\exp(Y))\rol(X) \rho_0(\exp(-Y))\rho_0(\exp(Y))p\\
&= (\exp(\rol(Y))\rol(X) \exp(-\rol(Y)))\rho_0(\exp(Y))p\\
&= \rol((\exp(Y)X \exp(-Y))\rho_0(\exp(Y))p\\
&= \rol((\Ad(\exp(Y)) X)\rho_0(\exp(Y))p,
\end{align*}
for all $p\in \pFock$ and $X,Y\in \mf g$.
From the decomposition in \cite[Theorem 6.4]{BC3} it immediately follows that $\pFock$ is locally $K$-finite. Using Proposition \ref{ExplicitK} we also see that $\pFock$ is also $K$-multiplicity finite.
\end{proof}

\begin{Cor}\label{CorFrech}
The $(\mf g, K)$-module $\pFock$ integrates to a unique smooth Fréchet representation of moderate growth for the Lie supergroup $\ds D(2,1;\alpha)$.
\end{Cor}
\begin{proof}
This follows immediately from \cite[Theorem 4.6]{Alldridge}.
\end{proof}

\subsection*{Acknowledgements}
SB is supported by a FWO postdoctoral junior fellowship from the Research Foundation Flanders (1269821N).

\bibliography{citations} 
\bibliographystyle{ieeetr}

\end{document}